\documentclass[12pt, reqno]{amsart}
\usepackage{amssymb, amsthm, amsmath, amsfonts}
\usepackage{array, epsfig}
\usepackage{bbm}
\usepackage{hyperref}
\usepackage[numbers,square]{natbib}
\usepackage{color}

  \evensidemargin
\oddsidemargin
\newcommand{\cF}{\mathcal{F}}

\newcommand{\cP}{\mathcal{P}}

\newcommand{\E}{\mathbb E}

\newcommand{\R}{\mathbb{R}}
\newcommand{\N}{\mathbb{N}}

\renewcommand{\P}{\mathbb{P}}

\newcommand{\Vol}{\mathop{\mathrm{Vol}}\nolimits}

\newcommand{\conv}{\mathop{\mathrm{conv}}\nolimits}
\newcommand{\pos}{\mathop{\mathrm{pos}}\nolimits}
\newcommand{\lin}{\mathop{\mathrm{lin}}\nolimits}
\newcommand{\aff}{\mathop{\mathrm{aff}}\nolimits}

\newcommand{\ind}{\mathbbm{1}}

\newcommand{\eee}{{\rm e}}

\newcommand{\Nor}{{\rm Nor}}

\theoremstyle{plain}
\newtheorem{theorem}{Theorem}[section]
\newtheorem{lemma}[theorem]{Lemma}

\theoremstyle{definition}

\theoremstyle{remark}
\newtheorem{remark}[theorem]{Remark}

\begin{document}

\author{Zakhar Kabluchko}
\address{Zakhar Kabluchko: Institut f\"ur Mathematische Stochastik,
Westf\"alische Wilhelms-Universit\"at M\"unster,
Orl\'eans-Ring 10,
48149 M\"unster, Germany}
\email{zakhar.kabluchko@uni-muenster.de}

\author{Christoph Th\"ale}
\address{Christoph Th\"ale: Fakult\"at f\"ur Mathematik,
Ruhr-Universit\"at Bochum,
44780 Bochum, Germany}
\email{christoph.thaele@rub.de}


\title[Projections of regular polytopes]{Monotonicity of expected $f$-vectors\\ for projections of regular polytopes}




\begin{abstract}
Let $P_n$ be an $n$-dimensional regular polytope from one of the three infinite series (regular simplices, regular crosspolytopes, and cubes). Project $P_n$ onto a random, uniformly distributed linear subspace of dimension $d\geq 2$.  We prove that the expected number of $k$-dimensional faces of the resulting random polytope is an increasing function of $n$. As a corollary, we show that the expected number of $k$-faces of the Gaussian polytope is an increasing function of the number of points used to generate the polytope. Similar results are obtained for the symmetric Gaussian polytope  and the Gaussian zonotope.
\end{abstract}

\maketitle

\section{Introduction}

Sample $n$ independent random points $X_1,\ldots,X_{n}$ from the standard Gaussian distribution on the $d$-dimensional Euclidean space $\R^d$. Their convex hull
$$
\cP_{n,d}
:=
\conv(X_1,\ldots,X_n) =
\left\{\sum_{i=1}^n \lambda_i X_i \colon \lambda_1,\ldots,\lambda_n \geq 0,  \sum_{i=1}^n \lambda_i = 1\right\}
$$
is called the \textit{Gaussian polytope}.

Given a deterministic or random polytope $P$ of dimension $\dim(P)$ we write $\cF_k(P)$ for the collection of its faces of dimension $k\in\{0,\ldots,\dim(P)\}$ ($k$-faces). The number of $k$-faces, i.e.\ the cardinality of the set $\cF_k(P)$, is denoted by $f_k(P)$. In particular, $f_0(P)$ is the number of vertices, $f_1(P)$ is the number of edges, $f_{\dim (P)-1}(P)$ is the number of facets of $P$, and $f_{\dim(P)}(P)=1$. We also put $f_k(P)=0$ for $k>\dim(P)$ and $k<0$.

In the present note we study what happens to $f_k(\cP_{n,d})$ if one further independent random point $X_{n+1}$ is added to the sample, i.e., if we pass from $\cP_{n,d}$ to $\cP_{n+1,d}$. If $X_{n+1}$ is contained in $\cP_{n,d}$, then $\cP_{n,d}=\cP_{n+1,d}$. However, if $X_{n+1}$ falls outside $\cP_{n,d}$, then $\cP_{n+1,d}$ is strictly larger than $\cP_{n,d}$ and the number of $k$-faces changes in a way which is difficult to control. For example, in general it is \textit{not} true that
$$
f_k(\cP_{n,d})  \leq f_k(\mathcal P_{n+1,d}) \quad \text{a.s.}
$$
Still, it is natural to conjecture that the expectation behaves in a monotone way, namely
\begin{equation}\label{eq:monotonicity_conj}
\E f_k(\mathcal P_{n,d}) \leq  \E f_k(\mathcal P_{n+1,d}).
\end{equation}
It should be pointed out that the expectations $\E f_k(\mathcal P_{n,d})$ cannot be accessed directly, since the probabilities that the random variables $f_k(\cP_{n,d})$ take some particular value are not known.

In her PhD thesis, Beermann~\cite[Theorem~5.3.1]{beermann_diss} proved~\eqref{eq:monotonicity_conj} for $k=d-1$ by an approach based on the Blaschke--Petkantschin formula~\cite[Chapter~7.2]{SW08}. Very recently,  Bonnet, Grote, Temesvari, Th\"ale, Turchi, Wespi~\cite{bonnet_etal} extended Beermann's method to convex hulls of i.i.d.\ vectors with certain non-Gaussian distributions. We refer to~\cite{bonnet_etal} for more on the history of the problem which started with the work of Devillers, Glisse, Goaoc, Moroz and Reitzner~\cite{devillers} inspired by a question of V. H.\ Vu~\cite{vu_sharp_concentration}.

The aim of the present note is to prove~\eqref{eq:monotonicity_conj} for all $k\in\{0,\ldots,d-1\}$ (even in the stronger form of a strict inequality).  From the work of Baryshnikov and Vitale~\cite{baryshnikov_vitale} it is known that, up to an affine transformation, the distribution of the Gaussian polytope $\cP_{n,d}$ coincides with that of the projection of a regular simplex with $n$ vertices onto a uniformly distributed $d$-dimensional linear subspace. As a consequence, the expected number of $k$-faces for $\cP_{n,d}$ is the same as for the projection of a regular simplex with $n$ vertices onto a random, uniformly distributed linear subspace of dimension $d$. The latter random projection is usually referred to as the \textit{Goodman--Pollack model}~\cite{AS92}. Thus, our question on the Gaussian polytope can be reduced to the corresponding question on the projection of the regular simplex. Similarly, projections of other two regular polytopes from the infinite series (the crosspolytope and the cube) correspond to the symmetric Gaussian polytope and the Gaussian zonotope.
Explicit and asymptotic formulae for the expected number of $k$-faces in the models mentioned above were provided in~\cite{AS92,baryshnikov_vitale,vershik_sporyshev,donoho_tanner2,donoho_tanner1,boroczky_henk,HMR04}.

To state our main result, let $P_n\subset \R^N$ be the $n$-dimensional regular polytope from one of the three infinite series, that is, $P_n$ is a regular simplex, a regular crosspolytope or a cube. The dimension $N$ of the  Euclidean space containing $P_n$ may be any number not less than $n$ and plays only a minor role in the following. Let $\Pi_d P_n$ denote the projection of $P_n$ onto a random $d$-dimensional linear subspace chosen according to the uniform distribution on the Grassmannian manifold. The next theorem states that the expected number of $k$-faces of $\Pi_d P_n$ strictly increases with $n$.

\begin{theorem}\label{theo:main}
 For any of the three infinite series of regular polytopes we have
$$
\E f_k(\Pi_d P_n) < \E f_k(\Pi_d P_{n+1})
$$
for all $n\geq 1$, $d \geq 2$, and $0 \leq k <  \min (n,d)$.
\end{theorem}

\begin{remark}\label{rem:non-strict}
Restrictions on $n,d,k$ are needed to make sure that the inequality is strict. The non-strict inequality holds for all $n\geq 1$, $d\geq 1$, $k\geq 0$. Indeed, for $d=1$ we trivially have $\E f_0(\Pi_1 P_n) =2= \E f_0(\Pi_1 P_{n+1})$ and $\E f_1(\Pi_1 P_n) =1= \E f_1(\Pi_1 P_{n+1})$, while  for $k>\min (n,d)$ we have $\E f_k(\Pi_d P_n)=0\leq \E f_k(\Pi_d P_{n+1})$. Finally, for $k=\min (n,d)$ we have $\E f_k(\Pi_d P_n)=1\leq \E f_k(\Pi_d P_{n+1})$.
\end{remark}


The rest of the paper is organized as follows. In Section~\ref{sec:Preliminaries} we introduce our notation and recall some facts about regular polytopes, internal and external angles as well as Grassmannians and intrinsic volumes. In Section~\ref{sec:projections} we  prove Theorem~\ref{theo:main}. Various corollaries and extensions (including the proof of~\eqref{eq:monotonicity_conj}) will be presented in Section~\ref{sec:extensions}.

\section{Some preliminaries}\label{sec:Preliminaries}

\subsection{Regular polytopes}\label{subsec:reg_poly}
We recall that a \textit{flag} of an $n$-dimensional polytope $P$ is a chain $F_0\subset F_1\subset\ldots \subset F_{n-1}$, where $F_k\in\cF_k(P)$ for all $k\in\{0,\ldots,n-1\}$. A polytope whose group of isometries acts transitively on the set of all its flags is usually called a \textit{regular polytope}; see \cite[Chapter 12.5]{BergerBook}. It is a well-known fact that there are precisely three infinite series of $n$-dimensional regular polytopes, the regular simplices, the regular crosspolytopes and the cubes. In the present paper we denote by $P_n$ the $n$-dimensional regular polytope from one of these three infinite series. Equivalently, we can define $P_n$ as an isometric copy of the following polytope:
\begin{equation}\label{eq:def_P_n}
\begin{cases}
\conv(e_1,\ldots,e_{n+1}) &\text{(regular simplex)},\\
\conv(e_1,\ldots,e_{n}, -e_1,\ldots, -e_n) &\text{(regular crosspolytope)},\\
[0,1]^n &\text{(cube)},
\end{cases}
\end{equation}
where $\conv(\,\cdot\,)$ denotes the convex hull and $e_1,e_2,\ldots$ are the vectors of the standard orthonormal basis.

\subsection{Internal and external angles}\label{subsec:angles}
The linear hull $\lin(A)$, the affine hull $\aff(A)$ and the positive hull $\pos(A)$ of a set $A$ are defined as the minimal linear subspace, affine subspace, or convex cone containing the set $A$. Roughly speaking, the angle of a convex cone $C$ in a Euclidean space is the fraction of the linear
subspace $\lin(C)$ occupied by $C$. More formally, the angle of $C$ is
defined as $\P[X_{\lin(C)} \in C]$, where $X_{\lin(C)}$ is a random vector with standard normal distribution on the linear  hull $\lin(C)$.

We shall next recall the concepts of external and internal angles of convex polytopes.
Let $P\subset \R^N$ be an $n$-dimensional polytope and let $G\in\cF_k(P)$ be a face of dimension $k\in\{0,\ldots,n\}$ of $P$. The relative interior of $G$ is the interior of $G$ taken with respect to the affine hull $\aff(G)$ as the ambient space. For a point $x$ in the relative interior of $G$ write $\Nor(P,x)$ for the convex cone of normal vectors of $P$ at $x$, that is
$$
\Nor(P,x) = \{u\in\R^N\colon \langle  u, p-x\rangle \leq 0 \text{ for all } p\in P\}.
$$
 Since $\Nor(P,x)$ is independent of the choice of $x$, we shall use the notation $\Nor(P,G)$ for this cone. The angle of this cone is called the \textit{external angle} $\gamma(G,P)$ of $P$ at its face $G$.


Next, we let $F$ be another face of $P$ such that $F\subset G$, that is, $F$ is a face of $G$.  For a point $x$ in the relative interior of $F$ we define the convex cone $A(x,G):=\pos(G-x)$. Again, this definition is independent of $x$ and we shall write $A(F,G)$ for this cone. Then, the \textit{internal angle} $\beta(F,G)$ of $F$ relative to $G$ is just the angle of this cone.
It is convenient to extend the definition to all pairs $(F,G)$ of faces of $P$ by putting $\beta(F,G):=0$ if $F$ is not a face of $G$.
We refer to~\cite[Chapter 14]{GruenbaumBook} for further details related to these concepts.

The internal and external angles are very difficult to compute in concrete situations. A rare exception is the case where $P=[0,1]^n$ is the $n$-dimensional cube. In this case, if $F\in\cF_k([0,1]^n)$, $G\in\cF_\ell([0,1]^n)$ and $F\subset G$  with $k\in\{0,\ldots,n\}$, $\ell\in\{k,\ldots,n\}$ then
$$
\gamma(G,P)=2^{-(n-\ell)}\qquad\text{and}\qquad\beta(F,G)=2^{-(\ell-k)}.
$$
An exact integral formula for the external angles of a regular simplex can be deduced from a volume formula for the regular \textit{spherical} simplex found in \cite{RubenTetrahedron} (see also \cite[Satz 6.5.3]{boehm_hertel_book} for a more accessible account). The asymptotic behaviour of the internal angles, as $n\to\infty$, is discussed in \cite{boroczky_henk}. For the regular crosspolytope we refer to \cite{baryshnikov_vitale,betke_henk,boroczky_henk} for exact and asymptotic formulas for the external and internal angles, respectively.

\subsection{Grassmannians and intrinsic volumes}\label{subsec:IntVol}
We let $G(n,k)$ be the Grassmannian of all $k$-dimensional linear subspaces of $\R^n$, $k\in\{0,\ldots,n\}$. It is well-known that $G(n,k)$ carries a natural structure of a smooth and compact $k(n-k)$-dimensional Riemannian manifold. The normalized $k(n-k)$-dimensional volume measure on $G(n,k)$ is denoted by $\nu_k$ and called the uniform distribution on $G(n,k)$; see~\cite[Chapter~5]{SW08}.

For an integer $\ell\geq 0$ denote the $\ell$-volume of the $\ell$-dimensional unit ball by
$$
\kappa_\ell:=\frac{\pi^{\ell/2}}{\Gamma\left(1+{\frac \ell 2}\right)}.
$$
The \textit{$k$-th intrinsic volume} $V_k(K)$ of a compact convex set $K\subset\R^n$ is defined as
$$
V_k(K) := \binom nk \frac{\kappa_n}{\kappa_k\kappa_{n-k}}\int_{G(n,k)}\Vol_k(K|L)\,\nu_k({\rm d} L),\qquad k\in\{0,\ldots,n\},
$$
where $\Vol_k(K|L)$ is the $k$-volume of the orthogonal projection $K|L$ of $K$ onto the linear subspace $L$. We emphasize that the intrinsic volumes are of outstanding importance in convex geometry since they form a basis of the vector space of all continuous (with respect to the usual Hausdorff distance), motion invariant valuations on the space of compact convex subsets of $\R^n$; see~\cite[Theorem 6.4.14]{SchneiderBook}. In particular, $V_0(K)=\ind_{\{K\neq\varnothing\}}$, $V_{1}(K)$ is a constant multiple of the mean width of $K$, $2V_{n-1}(K)$ is the same as the $(n-1)$-dimensional Hausdorff measure of the boundary of $K$ and $V_n(K)=\Vol_n(K)$ is the volume of $K$.

The $k$-th intrinsic volume of a polytope $P\subset\R^n$ admits the following neat interpretation in terms of the external angles introduced in Section~\ref{subsec:angles} at the $k$-dimensional faces of $P$ and their $k$-volumes:
\begin{equation}\label{eq:IntrinsicVolumePolytope}
V_k(P) = \sum_{F\in\cF_k(P)}\gamma(F,P)\,\Vol_k(F);
\end{equation}
see~\cite[Eq.~(4.23)]{SchneiderBook}.

The intrinsic volumes are monotone under set inclusion. That is, if $K$ and $K'$ are compact convex subsets of $\R^n$ with $K'\subseteq K$, then $V_k(K')\leq V_k(K)$ for all $k\in\{0,\ldots,n\}$. Moreover, the following strict form of monotonicity is known; see~\cite[Satz~16.7 on p.~185]{leichtweiss_book}  (where the result is stated for the so-called quermassintegrals $W_k$ which are constant multiples of the intrinsic volumes $V_{n-k}$, $k\in\{0,\ldots,n\}$). We say that a set $K\subset\R^n$ has dimension $d\in\{0,\ldots,n\}$ and write $\dim(K)=d$ if $d$ is the dimension of the affine hull of $K$.

\begin{lemma}\label{lem:Leichtweiss}
Let $K$ and $K'$ be compact convex subsets of $\R^n$ with $K'\subseteq K$ and $K\neq K'$. Then $V_k(K')=V_k(K)$ for some $k\in \{1,\ldots,n\}$ if and only if  $\dim(K)< k$.
\end{lemma}

\section{Projections of regular polytopes}\label{sec:projections}
\begin{proof}[Proof of Theorem~\ref{theo:main}]
We restrict ourselves to the case $d \leq  n$ since otherwise $f_k(\Pi_d P_n) = f_k(P_n)$, $f_k(\Pi_d P_{n+1}) = f_k(P_{n+1})$, and the required inequality becomes trivial.
So, let $d\geq 2$ and $0\leq k < d\leq n$.

Affentranger and Schneider~\cite[Eqn.~(4) on page~222]{AS92}, see also~\cite[Theorem 8.3.1]{SW08}, proved the following formula for the expected number of $k$-faces of a random projection of an arbitrary polytope $P$:
\begin{equation}\label{eq:affentranger_schneider}
\E f_k(\Pi_d P) = 2\sum_{s=0}^\infty \sum_{F\in \cF_k(P)} \sum_{G\in \cF_{d-2s-1}(P)} \beta(F,G) \gamma(G,P).
\end{equation}
We recall from Section~\ref{subsec:angles} that $\gamma(G,P)$ stands for the external angle of $G$, while $\beta(F,G)$ is the internal angle of $F$ relative to $G$.
By definition, the terms for which $F$ is not a face of $G$, vanish, as do the terms for which $d-2s-1<0$.

We use formula~\eqref{eq:affentranger_schneider} for $P=P_n$,  one of the regular $n$-dimensional polytopes defined in Section~\ref{subsec:reg_poly}. To prove Theorem~\ref{theo:main}, it suffices to show that  the quantity
\begin{equation}\label{eq:s_n_j}
s_n(j) := \sum_{F\in \mathcal F_k(P_n)} \sum_{G\in \mathcal F_{j-1}(P_n)} \beta(F,G) \gamma(G,P_n)
\end{equation}
is non-decreasing in $n$ for every $j:= d-2s\in \{1,\ldots,d\}$ and strictly increasing for $j=d$.

We notice that any $k$-face of a regular simplex (respectively, cube) is isometric to a $k$-dimensional regular simplex (respectively, cube). On the other hand, the $k$-faces of the $n$-dimensional regular crosspolytope are isometric copies of regular simplices, for $0\leq k\leq n-1$. 
Let $c_{m,\ell}$ be the number of $\ell$-faces of any $m$-dimensional face of $P_n$. In particular, $c_{n,\ell}$ is the number of $\ell$-faces of $P_n$ itself.
It is easy to provide an explicit formula for $c_{m,\ell}$, but we shall not need it.

The number of pairs $F\subset G$ such that $F\in \mathcal F_k(P_n)$ and $G\in \mathcal F_{j-1}(P_n)$ is given by $c_{n,j-1} c_{j-1,k}$. Indeed, first choosing $G$ as a face of $P_n$ contributes $c_{n,j-1}$, while choosing $F$ as a face of $G$  in a second step results in the factor $c_{j-1,k}$.
Since the polytope $P_n$ is regular, every pair $F\subset G$ as above can be transformed into any other such pair by an isometry of $P_n$.
In particular, every pair can be transformed into the ``canonical'' pair $Q_{k,n}\subset Q_{j-1,n}$, where
$$
Q_{i,n} =
\begin{cases}
\conv(e_1,\ldots,e_{i+1}) \subset P_n,  \qquad\qquad \text{for simplices and crosspolytopes},\\
\{(x_1,\ldots,x_{i}, 0,\ldots,0)\colon x_1,\ldots,x_i\in [0,1]\} \subset [0,1]^n, \qquad\,\, \text{ for cubes}.
\end{cases}
$$
Since isometries leave both internal and external angles invariant, it follows that
$$
s_n(j) = c_{n,j-1} c_{j-1,k} \beta(Q_{k,n},Q_{j-1,n}) \gamma(Q_{j-1,n},P_n).
$$
Observe that $j-1< d \leq n$ and the quantity $c_{j-1,k}\beta(Q_{k,n},Q_{j-1,n})$ does not depend on $n$. Also, it is non-zero for $j=d$. To prove the theorem, it therefore suffices to show that
\begin{equation}\label{eq:IntermediateStepProof}
c_{n,j-1} \gamma(Q_{j-1,n},P_n) \leq c_{n+1,j-1} \gamma(Q_{j-1,n+1},P_{n+1})
\end{equation}
for all $j\in\{1,\ldots,d\}$ with a strict inequality for $j\neq 1$.

Now, we multiply both sides of \eqref{eq:IntermediateStepProof} by the $(j-1)$-dimensional volume of $Q_{j-1,n}$ (which is isometric to $Q_{j-1, n+1}$). Note that $\text{Vol}_{j-1}(Q_{j-1,n})\neq 0$ because $j-1< d \leq n$.  We arrive at
$$
c_{n,j-1} \text{Vol}_{j-1}(Q_{j-1,n}) \gamma(Q_{j-1,n},P_n)  \leq c_{n+1,j-1} \text{Vol}_{j-1}(Q_{j-1,n+1}) \gamma(Q_{j-1,n+1},P_{n+1}).
$$
Recall that $\gamma(Q_{j-1,n},P_n)$ is the external angle of $P_n$ at any $(j-1)$-dimensional face, and that $c_{n,j-1}$ is the number of such faces. Therefore, using the representation \eqref{eq:IntrinsicVolumePolytope} of intrinsic volumes, the above inequality can be re-written as
$$
V_{j-1}(P_n) \leq V_{j-1} (P_{n+1})
$$
for all $j\in \{1,\ldots,d\}$ with a strict inequality for $j\neq 1$,  where $V_{j-1}$ is the $(j-1)$-st intrinsic volume. But we have $P_{n}\subset P_{n+1}$ and, as discussed in Section~\ref{subsec:IntVol}, the intrinsic volumes are monotone with respect to set inclusion. In fact, since the inclusion $P_{n}\subset P_{n+1}$ is strict and since $\dim(P_n)=n$ and $\dim(P_{n+1})=n+1>d > j-1$, we even have $V_{j-1}(P_n) < V_{j-1} (P_{n+1})$ for all $j\in \{2,\ldots,d\}$ by Lemma~\ref{lem:Leichtweiss}. This completes the argument.
\end{proof}

\section{Corollaries and extensions}\label{sec:extensions}

\subsection{Gaussian versions}
Let $X_1,\ldots,X_n$ be independent random vectors with standard Gaussian distribution on $\R^d$.
Recall that the Gaussian polytope $\cP_{n,d}$ is defined as the convex hull of $X_1,\ldots,X_n$.
As discussed in the introduction, from the work of Baryshnikov and Vitale~\cite{baryshnikov_vitale} it is known that $\cP_{n,d}$ has, up to an affine transformation, the same distribution as the Goodman--Pollack polytope $\Pi_d P_{n-1}$, where $P_{n-1}$ is a regular simplex with $n$ vertices. As a consequence,
$$
\E f_k(\cP_{n,d})=\E f_k(\Pi_d P_{n-1}) \text { for all } k\in\{0,\ldots,d-1\},
$$
which leads to the following corollary of Theorem~\ref{theo:main}:

\begin{theorem}\label{theo:main_gauss_poly}
We have $\E f_k(\mathcal P_{n,d}) <  \E f_k(\mathcal P_{n+1,d})$ for all $n\geq 1$, $d \geq 2$, and $0 \leq k <  \min (n-1,d)$.
\end{theorem}

We emphasize that Theorem~\ref{theo:main_gauss_poly} generalizes Beermann's result~\cite[Theorem 5.3.1]{beermann_diss} and its strengthened form~\cite[Theorem~1]{bonnet_etal} for the case $k=d-1$ to faces of arbitrary dimension.

\medskip

Similarly, if $P_n$ is the $n$-dimensional regular crosspolytope, then the expected number of $k$-dimensional faces of $\Pi_d P_n$ coincides with the expected number of $k$-faces of the  \textit{symmetric Gaussian polytope} $\cP^{\rm sym}_{n,d}$ defined by
$$
\cP^{\rm sym}_{n,d} = \conv (X_1,\ldots,X_n,-X_1,\ldots,-X_n).
$$
Although this fact is not explicitly stated in~\cite{baryshnikov_vitale}, it can be proved just by repeating the argument of~\cite{baryshnikov_vitale} (also compare with the remark after Corollary 1.1 in~\cite{boroczky_henk}). We can thus draw from Theorem~\ref{theo:main} the following conclusion:

\begin{theorem}\label{theo:main_gauss_poly_symm}
We have $\E f_k(\mathcal P^{\rm sym}_{n,d}) < \E f_k(\mathcal P^{\rm sym}_{n+1,d})$ for all $n\geq 1$, $d \geq 2$, and $0 \leq k <  \min (n,d)$.
\end{theorem}

Finally, we consider the \textit{Gaussian zonotope} $\mathcal Z_{n,d}$, which is defined to be the Minkowski sum of the segments $[0,X_1], \ldots, [0,X_n]$, that is,
$$
\mathcal Z_{n,d} = \left\{\sum_{i=1}^n \lambda_i X_i \colon \lambda_1,\ldots,\lambda_n \in [0,1]\right\}.
$$
The expected number of $k$-faces of $\mathcal Z_{n,d}$ is the same as the expected number of $k$-faces of $\Pi_d P_n$, where $P_n$ is an $n$-dimensional cube, a fact mentioned in~\cite[p.~146]{baryshnikov_vitale}. As above, Theorem~\ref{theo:main} yields the following result:

\begin{theorem}\label{theo:main_gauss_zonotope}
We have $\E f_k(\mathcal Z_{n,d}) < \E f_k(\mathcal Z_{n+1,d})$ for all $n\geq 1$, $d \geq 2$, and $0 \leq k <  \min (n,d)$.
\end{theorem}

\subsection{Poissonized versions}
Let us replace the fixed number $n$ of points by a Poisson random variable $N(t)$ with parameter $t>0$, that is,
$$
\P[N(t)=\ell]=\eee^{-t} \frac{t^\ell}{\ell!}, \quad \ell=0,1,\ldots
$$
We assume that $N(t)$ is independent of $X_1,X_2,\ldots$.  Then, it is natural to conjecture that the expected number of $k$-faces is monotone in $t$ for all three models, the Gaussian polytopes, the symmetric Gaussian polytopes and the Gaussian zonotopes. For faces of maximal dimension of Gaussian polytopes, the next theorem was proved by Beermann~\cite[Theorem~5.4.1]{beermann_diss}.

\begin{theorem}
For all $d\in\N$ and $0\leq k \leq  d$ the functions $t\mapsto \E f_k (\mathcal P_{N(t), d})$,  $t\mapsto \E f_k (\cP^{\rm sym}_{N(t), d})$ and $t\mapsto \E f_k (\mathcal Z_{N(t), d})$ are non-decreasing in $t>0$.
\end{theorem}
\begin{proof}
We only prove the result for the Gaussian polytopes, since for the other two models the arguments are similar.

The proof is based on a so-called coupling argument. To start with, we let $\{N(t) \colon t\geq 0\}$ be a homogeneous Poisson process with intensity $1$ which is independent of everything else. In particular, for each $t>0$, $N(t)$ is a Poisson random variable with parameter $t$. Take some $0 < t_1 <t_2$. Then, $N(t_1) \leq N(t_2)$ a.s.\ For integers $0\leq n_1 \leq n_2$ consider the event $A(n_1,n_2) := \{N(t_1) = n_1, N(t_2) = n_2\}$. Since the Poisson process is independent of everything else, we have
$$
\E [f_k (\mathcal P_{N(t_i), d})\ind_{A(n_1,n_2)}]
=
\E [f_k (\mathcal P_{n_i, d})\ind_{A(n_1,n_2)}]
=
\E [f_k (\mathcal P_{n_i, d})] \P[A(n_1,n_2)]
$$
for $i=1,2$. We already know from Theorem~\ref{theo:main_gauss_poly} (see also Remark~\ref{rem:non-strict}) that $\E [f_k (\mathcal P_{n_1, d})] \leq \E [f_k (\mathcal P_{n_2, d})]$, whence
$$
\E [f_k (\mathcal P_{N(t_1) , d})\ind_{A(n_1,n_2)}] \leq \E [f_k (\mathcal P_{N(t_2), d})\ind_{A(n_1,n_2)}]
$$
for all $0\leq n_1\leq n_2$. Taking the sum over all such  $n_1,n_2$ yields
\begin{align*}
\E f_k(\cP_{N(t_1),d})
&=
\sum_{\substack{(n_1,n_2)\in\N^2_0\\ n_1\leq n_2}} \E [f_k (\mathcal P_{N(t_1), d})\ind_{A(n_1,n_2)}]\\
&\leq
\sum_{\substack{(n_1,n_2)\in\N^2_0\\ n_1\leq n_2}} \E [f_k (\mathcal P_{N(t_2), d})\ind_{A(n_1,n_2)}] = \E f_k(\cP_{N(t_2),d}),
\end{align*}
and the result is proved.
\end{proof}

\subsection{Other functionals of faces}
Let $P$ be a polytope.
For the Gaussian polytope in $\R^d$ Hug, Munsonius and Reitzner~\cite[Section 3.2]{HMR04} considered the functional
$$
T^{d,k}_{0,b} (P) = \sum_{F\in \cF_k(P)} (\text{Vol}_k(F))^b, \quad b\geq 0, \quad 0\leq k \leq \dim P,
$$
which reduces to $f_{k} (P)$ for $b=0$ and to the total $k$-dimensional volume of all $k$-faces of $P$ for $b=1$. In particular, $T_{0,1}^{d,d-1}(P)=2V_{d-1}(P)$ is the total surface area of $P$ and $T_{0,b}^{d,d}  = \Vol_d^b(P)$, but the other functionals of this form cannot be expressed in terms of intrinsic volumes of $P$. With the aim of the Blaschke--Petkantschin formula, they proved that
$$
\E T^{d,k}_{0,b} (\cP_{n,d}) = \E f_k ( \cP_{n,d})\, \cdot  \left(\frac{\sqrt {k+1}}{k!}\right)^b \prod_{j=1}^k \frac{\Gamma\left(\frac{d+b+1-j}{2}\right)}{\Gamma\left(\frac{d+1-j}{2}\right)}.
$$
While the (almost sure) monotonicity of $T_{0,1}^{d,d-1}(\cP_{n,d})$ in $n$ is a direct consequence of the monotonicity of the intrinsic volumes under set inclusion, the following non-trivial monotonicity result for the mean values of $T^{d,k}_{0,b} (\cP_{n,d})$ is a consequence of Theorem~\ref{theo:main_gauss_poly}:

\begin{theorem}
Let $b\geq 0$. We have $\E T^{d,k}_{0,b} (\cP_{n,d}) <  \E T^{d,k}_{0,b} (\cP_{n+1,d})$ for all $n\geq 1$, $d \geq 2$, and $0 \leq k <  \min (n-1,d)$.
\end{theorem}

A similar result can be obtained for the symmetric Gaussian polytope as well.

\section*{Acknowledgement}
Z.K.\ is grateful to Dmitry Zaporozhets for numerous discussions related to the topic of the paper.

This work was initiated during the Oberwolfach mini-workshop \textit{Perspectives in High-dimensional Probability and Convexity}. All support is gratefully acknowledged.

\bibliography{monotonicity_bib}
\bibliographystyle{plainnat}

\end{document}